\documentclass[11pt]{amsart}

\usepackage[T1]{fontenc}
\usepackage[utf8]{inputenc}
\usepackage{lmodern}
\usepackage{amsmath,amssymb,mathtools,amscd}
\usepackage{enumitem}
\usepackage{microtype}
\usepackage{hyperref}
\hypersetup{
  hidelinks,
  unicode=true,
  pdftitle={Iterated Cartier Flux and Non-Finite Generation of Tame Polynomial Automorphism Groups over Finite Fields},
  pdfauthor={Stefan Barańczuk and Tomasz Ślusarski},
  pdfsubject={Tame polynomial automorphism groups over finite fields and iterated Cartier-flux characters},
  pdfkeywords={tame polynomial automorphism, finite field, finite generation, Cartier operator, algebraic de Rham cohomology, mod-p abelianization}
}
\usepackage{aliascnt}
\usepackage[nameinlink,noabbrev]{cleveref}

\newtheorem{theorem}{Theorem}[section]

\newaliascnt{proposition}{theorem}
\newtheorem{proposition}[proposition]{Proposition}
\aliascntresetthe{proposition}

\newaliascnt{lemma}{theorem}
\newtheorem{lemma}[lemma]{Lemma}
\aliascntresetthe{lemma}

\newaliascnt{corollary}{theorem}
\newtheorem{corollary}[corollary]{Corollary}
\aliascntresetthe{corollary}

\newaliascnt{conjecture}{theorem}
\newtheorem{conjecture}[conjecture]{Conjecture}
\aliascntresetthe{conjecture}

\theoremstyle{definition}
\newaliascnt{definition}{theorem}
\newtheorem{definition}[definition]{Definition}
\aliascntresetthe{definition}

\newaliascnt{example}{theorem}
\newtheorem{example}[example]{Example}
\aliascntresetthe{example}

\newaliascnt{problem}{theorem}
\newtheorem{problem}[problem]{Open Problem}
\aliascntresetthe{problem}

\theoremstyle{remark}
\newaliascnt{remark}{theorem}
\newtheorem{remark}[remark]{Remark}
\aliascntresetthe{remark}

\crefname{theorem}{Theorem}{Theorems}
\Crefname{theorem}{Theorem}{Theorems}
\crefname{proposition}{Proposition}{Propositions}
\Crefname{proposition}{Proposition}{Propositions}
\crefname{lemma}{Lemma}{Lemmas}
\Crefname{lemma}{Lemma}{Lemmas}
\crefname{corollary}{Corollary}{Corollaries}
\Crefname{corollary}{Corollary}{Corollaries}
\crefname{conjecture}{Conjecture}{Conjectures}
\Crefname{conjecture}{Conjecture}{Conjectures}
\crefname{definition}{Definition}{Definitions}
\Crefname{definition}{Definition}{Definitions}
\crefname{example}{Example}{Examples}
\Crefname{example}{Example}{Examples}
\crefname{problem}{Open Problem}{Open Problems}
\Crefname{problem}{Open Problem}{Open Problems}
\crefname{remark}{Remark}{Remarks}
\Crefname{remark}{Remark}{Remarks}

\DeclareMathOperator{\Aut}{Aut}
\DeclareMathOperator{\Spec}{Spec}

\DeclareMathOperator{\im}{im}
\DeclareMathOperator{\spanop}{span}
\DeclareMathOperator{\pr}{pr}

\newcommand{\Fp}{\mathbf F_p}
\newcommand{\Fq}{\mathbf F_q}
\newcommand{\A}{\mathbf A}
\newcommand{\N}{\mathbf N}
\newcommand{\id}{\mathrm{id}}

\newcommand{\DAMW}{\operatorname{DA}^{\mathrm{MW}}}
\newcommand{\jet}{\operatorname{jet}}

\title[Iterated Cartier Flux]{Iterated Cartier Flux and Non-Finite Generation of Tame Polynomial Automorphism Groups over Finite Fields}

\author{Stefan Barańczuk}
\thanks{Corresponding author.}
\address{Faculty of Mathematics and Computer Science, Adam Mickiewicz University, Poznań, Poland}
\email{stefbar@amu.edu.pl}
\urladdr{https://orcid.org/0000-0003-4198-9241}

\author{Tomasz Ślusarski}
\address{}
\email{tomasz.slusarski@amu.edu.pl}

\date{August 24, 2026}

\begin{document}

\begin{abstract}
Let $k=\mathbf F_q$ be a finite field of characteristic $p$, and let
\[
U_n(k)=\{F\in\operatorname{TA}_n(k):F(0)=0,\ JF(0)=I_n\}.
\]
We prove that every special tame polynomial automorphism is repeatedly
Cartier-admissible for the de Rham flux class associated with
\[
\lambda=x_1\,dx_2\wedge\cdots\wedge dx_n.
\]
Lowest-weight projection of the iterated Cartier descents gives additive
characters
\[
\chi_{r,j}:U_n(k)\longrightarrow(k,+),
\qquad r\ge0,\quad 1\le j\le n,
\]
and for every fixed group element all but finitely many of these characters
vanish. Put
\[
S_{n,p}=\begin{cases}
\mathbf N_{\ge1},&(n,p)=(2,2),\\
\mathbf N_0,&\text{otherwise}.
\end{cases}
\]
The joint character map is surjective onto
\[
\bigoplus_{s\in S_{n,p}}(k^n,+),
\]
and every fixed-coordinate submap onto
\[
\bigoplus_{s\in S_{n,p}}(k,+)
\]
admits an explicit group-theoretic section. Consequently,
\[
\dim_{\mathbf F_p}
\frac{U_n(k)}{[U_n(k),U_n(k)]\,U_n(k)^{[p]}}
=\aleph_0.
\]
The character $\chi_{r,j}$ factors through the jet of order
\[
N_r=(n-1)(p^{r+1}-1),
\]
and this order is optimal for every $r\in S_{n,p}$. Since $U_n(k)$ has
finite index in $\operatorname{TA}_n(k)$, the tame polynomial automorphism
group over a finite field is finitely generated exactly in dimension one.
In particular, this proves the Maubach--Willems finite-generation conjecture
for $n\ge3$.
\end{abstract}

\maketitle

\section{Introduction}

Let $k$ be a field. The tame polynomial automorphism group
$\operatorname{TA}_n(k)$ is generated by the affine automorphisms of
$\A_k^n$ and the elementary automorphisms
\[
E_i(f)=(x_1,\ldots,x_{i-1},x_i+f,x_{i+1},\ldots,x_n),
\qquad
f\in k[x_1,\ldots,\widehat{x_i},\ldots,x_n].
\]
In characteristic zero and dimension at least three, Derksen proved that the
affine subgroup together with one suitable nonlinear elementary
automorphism generates the entire tame group; see
\cite[Section~5.2]{vanDenEssen2000}. Bodnarchuk subsequently proved that an
arbitrary nonlinear triangular automorphism may be used in place of
Derksen's original generator \cite{Bodnarchuk2005}.

The finite-field situation is different. Let $k=\Fq$, where $q=p^a$. In a
preprint first posted in 2009 and published in 2011, Maubach and Willems
proved that, for $n\ge3$, the group $\operatorname{TA}_n(k)$ is generated by
the affine subgroup together with the infinite family
\begin{equation}
\left(x_1+x_2^{k_2p-1}\cdots x_n^{k_np-1},x_2,\ldots,x_n\right),
\qquad
1\le k_2\le\cdots\le k_n,
\end{equation}
and formulated the following conjecture
\cite[Conjecture~2.1 and Theorem~2.3]{MaubachWillems2011}.

\begin{conjecture}[Maubach--Willems]
\label{conj:MW}
Let $k$ be a finite field and let $n\ge3$. There is no finite subset
$\mathcal E\subseteq\operatorname{TA}_n(k)$ such that
\[
\operatorname{TA}_n(k)=\langle\operatorname{Aff}_n(k),\mathcal E\rangle.
\]
\end{conjecture}

Since $\operatorname{Aff}_n(k)$ is finite, \cref{conj:MW} is equivalent to
non-finite generation of $\operatorname{TA}_n(k)$. Kuroda later recorded the
same finite-generation problem in work on stable co-tameness
\cite[p.~1032]{Kuroda2017}.

Kuroda's bibliography also lists Motoki Kuroda's 2016 Japanese master's
thesis \emph{Stably Derksen polynomial automorphisms over finite fields}.
The thesis was not obtainable for inspection. We therefore use only the
published account in \cite[Section~6.3]{Kuroda2017}: it reports, for
$n\ge3$, that $E_1(x_2^2)$ is stably co-tame over $\mathbf F_p$ for $p\ge3$,
that stable co-tameness of $E_1(x_2^5)$ over $\mathbf F_3$ was difficult to
decide, and that $E_1(x_2^5)$ is stably co-tame over $\mathbf F_9$. These
statements concern stable co-tameness of individual automorphisms; they
neither assert non-finite generation of the full tame group nor construct
the characters below. To the best of our knowledge, \cref{cor:intro-classification}
gives the first proof of \cref{conj:MW}.

Maubach and Willems also introduced the modified finite-field Derksen
subgroup
\begin{equation}
\DAMW_n(k)=
\left\langle
\operatorname{Aff}_n(k),
\left(x_1+x_2^{p-1}\cdots x_n^{p-1},x_2,\ldots,x_n\right)
\right\rangle
\qquad(n\ge3).
\end{equation}
They proved that, for every $m\ge1$, its permutation image on
$\mathbf F_{q^m}^n$ is the same as the corresponding image of the full tame
group \cite[Theorem~3.1]{MaubachWillems2011}. Thus finite-extension
permutation images do not distinguish the two groups. This is related to
the profinite viewpoint developed by Maubach and Rauf
\cite{MaubachRauf2015}.

This modified subgroup must be distinguished from the original Derksen
group
\[
\operatorname{DA}_n(k)=\langle\operatorname{Aff}_n(k),E_1(x_2^2)\rangle.
\]
In characteristic two, Hakuta proved
$\operatorname{DA}_n(\mathbf F_2)\subsetneq\operatorname{TA}_n(\mathbf F_2)$
for every $n\ge3$ \cite[Theorem~1]{Hakuta2019}; he subsequently proved, over
finite fields of characteristic two under the hypotheses of that paper,
that the permutation group induced by the original Derksen group is a
proper subgroup of the corresponding alternating group \cite{Hakuta2021}.
These results concern a different one-generator subgroup and
permutation-theoretic obstructions. They do not imply the countable
quotient or the finite-generation theorem proved here.

Our first result gives substantially more than non-finite generation. Put
\[
U_n(k)=\{F\in\operatorname{TA}_n(k):F(0)=0,\ JF(0)=I_n\}.
\]
For a group $G$, let
\[
G^{[p]}=\langle g^p:g\in G\rangle.
\]
Then
\[
G/[G,G]G^{[p]}
\]
is the universal elementary abelian $p$-quotient of $G$, also called its
mod-$p$ abelianization. Direct sums below consist of finitely supported
families. Define
\begin{equation}
S_{n,p}=
\begin{cases}
\N_{\ge1},&(n,p)=(2,2),\\
\N_0,&\text{otherwise}.
\end{cases}
\end{equation}

\begin{theorem}[Cartier character quotients]
\label{thm:main}
Let $k=\Fq$ have characteristic $p$, and let $n\ge2$. There are group
homomorphisms
\[
\chi_{r,j}:U_n(k)\longrightarrow(k,+),
\qquad r\ge0,\quad 1\le j\le n,
\]
with the following properties.
\begin{enumerate}[label=\textup{(\roman*)}]
\item For each $F\in U_n(k)$, all but finitely many $\chi_{r,j}(F)$ vanish.

\item The joint homomorphism
\[
\chi_{\mathrm{all}}:U_n(k)\longrightarrow
\bigoplus_{s\in S_{n,p}}(k^n,+)
\]
is surjective.

\item For each fixed $1\le j\le n$, the homomorphism
\[
\chi_{S,j}:U_n(k)\longrightarrow
\bigoplus_{s\in S_{n,p}}(k,+),
\qquad
F\longmapsto(\chi_{s,j}(F))_{s\in S_{n,p}},
\]
is a split epimorphism.

\item The mod-$p$ abelianization
\[
Q_n(k)=
\frac{U_n(k)}{[U_n(k),U_n(k)]\,U_n(k)^{[p]}}
\]
satisfies
\[
\dim_{\Fp}Q_n(k)=\aleph_0.
\]
\end{enumerate}
\end{theorem}

Each $\chi_{r,j}$ already factors through a finite jet: the required order
is
\[
N_r=(n-1)(p^{r+1}-1),
\]
and this order is optimal for every $r\in S_{n,p}$; see
\cref{prop:finite-jet,cor:optimal-jet}.

\begin{corollary}[Finite-generation classification]
\label{cor:intro-classification}
For every finite field $k$ and every $n\ge1$,
\[
\operatorname{TA}_n(k)\text{ is finitely generated}
\quad\Longleftrightarrow\quad
n=1.
\]
In particular, \cref{conj:MW} holds.
\end{corollary}

The Cartier construction applies uniformly in dimension two and yields the
explicit quotient and countably infinite-dimensional mod-$p$
abelianization stated in \cref{thm:main}; no separate plane-group
decomposition is used.

The proof has the following architecture:
\[
\begin{aligned}
\text{flux}
&\longrightarrow
\text{Cartier-admissibility}
\longrightarrow
\text{lowest-weight characters}\\
&\longrightarrow
\text{delta evaluation}
\longrightarrow
\text{quotient}
\longrightarrow
\text{non-finite generation}.
\end{aligned}
\]
The polynomial de Rham calculation, relative Cartier isomorphism, and
primitive-form crossed cocycle are foundational inputs. The same formal
Cartier--flux construction also occurs in recent work of Ma on finite
congruence quotients of the ambient Jacobian-one formal automorphism group.
The new contribution here begins with repeated Cartier descent on the
special tame polynomial subgroup. We prove that every special tame element
is Cartier-admissible, extract ordinary characters at arbitrarily deep
positive iterated levels, and construct a Frobenius-normalized family of
elementary shears with cross-level Kronecker-delta values. The joint
character map is surjective onto a countable direct sum, every
fixed-coordinate submap is split, and the resulting maximal elementary
abelian $p$-quotient has countably infinite dimension. This yields the
finite-generation classification of tame polynomial automorphism groups
over finite fields.

The roles of the standing hypotheses are distinct. Perfectness of the
finite field supplies the inverse Frobenius on coefficients and the
displayed polynomial de Rham decomposition. Tameness is used to prove
Cartier-admissibility from special affine and elementary generators.
Finiteness is used for the finite-index reduction, the finiteness of
truncated jet images, countability, and the final finite-generation
classification.

During preparation of this manuscript, we became aware of Ma's
July--August 2026 preprints on finite congruence quotients $T_M$ of the full
tangent-to-the-identity Jacobian-one formal automorphism group over
$\mathbf F_p$. In the common range $n\ge3$ and $p\ge5$, Ma uses the same
primitive
\[
\lambda=x_1\,dx_2\wedge\cdots\wedge dx_n
\]
and the same crossed flux class $[F^*\lambda-\lambda]$. His Cartier--Frattini coordinate theorem identifies the Frattini subgroup
with the common kernel of the quadratic coordinate and the retained constant,
trace, and divergence coordinates $Q_0,Q_1,Q_2$; equivalently, no visible
Cartier layer of index at least three contributes to the Frattini quotient
\cite[Theorem~5.2 and Corollary~5.5]{MaCartierFrattini2026}. Under the
standard Cartier--contraction identification, the restriction of $Q_0$ to
$U_n(\mathbf F_p)$ agrees with our zeroth vector character
$\chi_0=(\chi_{0,1},\ldots,\chi_{0,n})$.

More precisely, if
\[
d_t=(p-1)(n-1)+pt,
\qquad
q_r=(n-1)(p^r-1),
\]
then our optimal jet degree satisfies $N_r=d_{q_r}$. Moreover, over
$\mathbf F_p$, the test shear $v_{r,j}(a)$ is the one-coordinate
inverse-Cartier monomial representative obtained by taking
$\alpha_j=0$ and $\alpha_i=p^r-1$ for $i\ne j$ in
\cite[Lemma~3.6]{MaCartierFrattini2026}. The conclusions diverge at
positive levels: for $r\ge1$, the corresponding ambient Cartier component
is killed in the Frattini quotient of $T_M$, whereas its restriction to the
tame polynomial subgroup supports the nonzero character $\chi_{r,j}$.
Thus these positive-depth characters do not extend to the corresponding
ambient finite formal congruence group; this is made precise in
\cref{prop:Ma-comparison}.

Ma's relation-symbol paper constructs a canonical surjection
$\bar\pi:R_3\twoheadrightarrow K_3$ with an explicit Euler lift, shows that
the homology of a divergence-free coefficient subcomplex maps onto the
remaining kernel, and proves survival of the canonical $K_3$ cross-term
under the map from $H_2(T_M,\mathbf F_p)$ to
$H_2(T_M/\Phi(T_M),\mathbf F_p)$
\cite[Lemma~2.5, Section~3.3, and Theorem~4.10]{MaRelationSymbols2026}.
The subsequent Frattini--commutator paper determines the first
Frattini--commutator layer, identifies it with $P_3/P_4$, and proves a lower
$p$-central Cartier support cone
\cite[Theorem~3.18, Proposition~3.25, and Proposition~4.7]{MaFrattiniCommutator2026}. Those homological and commutator-layer
theorems are not addressed here, and neither they nor the ambient
Frattini calculation imply the positive-depth tame-subgroup quotient or
the finite-generation theorem. No theorem from these preprints is used as
an input to the proofs of our principal character, quotient, or
finite-generation results. The comparison proposition below uses Ma's
Frattini theorem only to identify the ambient non-extension phenomenon.
The foundational Cartier--flux machinery is shared, but the decisive
iterative tame-subgroup argument and the proof-ending mechanism are
different.

The paper is organized as follows. \Cref{sec:groups} fixes the group
convention and establishes the finite-index reduction. \Cref{sec:cartier}
computes polynomial de Rham cohomology and records relative Cartier
naturality. \Cref{sec:flux} develops the flux cocycle and proves the
structural tame Cartier-admissibility theorem. \Cref{sec:characters}
constructs the iterated lowest-weight characters and proves finite-jet
dependence. \Cref{sec:quotient} evaluates them on explicit elementary
shears, proves the delta and quotient theorems, and gives the precise
comparison with Ma's ambient formal coordinates before proving
\cref{thm:main,cor:intro-classification}. \Cref{sec:open} records the
limitations of the construction and several open problems.

\section{Tame automorphism groups and the tangent subgroup}
\label{sec:groups}

\subsection{Conventions and special tame generators}

We use the geometric convention
\[
\operatorname{GA}_n(k)=\operatorname{Aut}_k(\A_k^n).
\]
An element is represented by a coordinate tuple $F=(F_1,\ldots,F_n)$, and multiplication is geometric composition,
\[
FG=F\circ G.
\]
The induced pullback on the coordinate ring $R=k[x_1,\ldots,x_n]$ is determined by $F^*(x_i)=F_i$ and satisfies
\[
(FG)^*=G^*F^*.
\]
Thus the coordinate-ring automorphism group is the opposite group to the geometric convention used here.

The affine subgroup is denoted by $\operatorname{Aff}_n(k)$. The tame group $\operatorname{TA}_n(k)$ is generated by $\operatorname{Aff}_n(k)$ and all elementary automorphisms. Let $\operatorname{SAff}_n(k)$ be the affine automorphisms whose linear part has determinant one, and put
\[
\operatorname{STA}_n(k)
=
\{F\in\operatorname{TA}_n(k):\det JF=1\}.
\]

\begin{lemma}[Constant Jacobian determinant]
Let $k$ be a field and let $F\in\operatorname{GA}_n(k)$. Then
\[
\det JF\in k^\times.
\]
If $JF(0)=I_n$, then $\det JF=1$.
\end{lemma}

\begin{proof}
Let $G=F^{-1}$. The chain rule gives
\[
JG(F(x))JF(x)=I_n.
\]
Taking determinants yields
\[
\det JG(F(x))\det JF(x)=1
\]
in $k[x_1,\ldots,x_n]$. Hence $\det JF$ is a unit of the polynomial ring and therefore belongs to $k^\times$. If $JF(0)=I_n$, evaluation at the origin gives $\det JF=1$.
\end{proof}

\begin{proposition}[Generation of the special tame group]
\label{prop:special-generation}
For every field $k$, the group $\operatorname{STA}_n(k)$ is generated by $\operatorname{SAff}_n(k)$ and the elementary automorphisms.
\end{proposition}

\begin{proof}
Let $H$ be the subgroup generated by $\operatorname{SAff}_n(k)$ and the elementary automorphisms. Every generator has Jacobian determinant one, so
\[
H\subseteq\operatorname{STA}_n(k).
\]

For $c\in k^\times$, put
\[
D_c=(cx_1,x_2,\ldots,x_n),
\qquad
D=\{D_c:c\in k^\times\}.
\]
Conjugation by $D_c$ preserves $\operatorname{SAff}_n(k)$. It also preserves elementary automorphisms. Indeed,
\[
D_cE_1(f)D_c^{-1}=E_1(cf),
\]
whereas, for $i\ne1$,
\[
D_cE_i(f)D_c^{-1}
=
E_i\bigl(f(c^{-1}x_1,x_2,\ldots,x_n)\bigr).
\]
The polynomial on the right remains independent of $x_i$. Thus $D$ normalizes $H$, and $DH$ is a subgroup.

Let $A$ be affine, and let $c$ be the determinant of its linear part. Then
\[
A=D_cB,
\qquad
B=D_c^{-1}A\in\operatorname{SAff}_n(k)\subseteq H.
\]
Every affine generator belongs to $DH$, and every elementary generator belongs to $H$. Hence
\[
\operatorname{TA}_n(k)=DH.
\]

If $F\in\operatorname{STA}_n(k)$, write $F=D_ch$ with $h\in H$. Then
\[
1=\det JF=c,
\]
so $F=h\in H$. Therefore $\operatorname{STA}_n(k)=H$.
\end{proof}

The preceding proposition concerns $\operatorname{STA}_n(k)$. No generation
statement for the smaller tangent subgroup $U_n(k)$ is used in the sequel.

\subsection{Finite index and boundary cases}

\begin{proposition}[Finite index of the tangent subgroup]
\label{prop:finite-index}
Let $k=\Fq$. Then $U_n(k)$ has finite index in $\operatorname{TA}_n(k)$, and
\[
[\operatorname{TA}_n(k):U_n(k)]
\le q^n|\operatorname{GL}_n(k)|.
\]
\end{proposition}

\begin{proof}
Let
\[
G_0=\{F\in\operatorname{TA}_n(k):F(0)=0\}.
\]
The group $\operatorname{TA}_n(k)$ acts on the finite set $k^n$, and $G_0$ is the stabilizer of the origin. Hence
\[
[\operatorname{TA}_n(k):G_0]\le q^n.
\]

For $F,G\in G_0$, geometric composition and the chain rule give
\[
J(FG)(0)=JF(G(0))JG(0)=JF(0)JG(0).
\]
Therefore
\[
J_0:G_0\longrightarrow\operatorname{GL}_n(k),
\qquad
F\longmapsto JF(0),
\]
is a homomorphism. Its kernel is $U_n(k)$, so
\[
[G_0:U_n(k)]=|\im J_0|\le|\operatorname{GL}_n(k)|.
\]
Multiplying the two index bounds proves the result.
\end{proof}

\begin{lemma}[Schreier]
\label{lem:Schreier}
Let $G$ be finitely generated and let $H\le G$ have finite index. Then $H$ is finitely generated.
\end{lemma}

\begin{proof}
Choose a finite symmetric generating set $S$ for $G$, and choose a finite set $T$ of representatives for the right cosets $H\backslash G$, with $1\in T$. For $g\in G$, let $\overline g\in T$ denote the representative of $Hg$.

The finite set
\[
\mathcal S_H=
\{ts\,\overline{ts}^{-1}:t\in T,\ s\in S\}
\]
is contained in $H$. Let
\[
h=s_1\cdots s_m\in H,
\qquad s_i\in S,
\]
and put
\[
t_i=\overline{s_1\cdots s_i},
\qquad t_0=1.
\]
Since $h\in H$, one has $t_m=1$. The product telescopes:
\[
h=
\prod_{i=1}^m
\bigl(t_{i-1}s_it_i^{-1}\bigr).
\]
Every factor belongs to $\mathcal S_H$, so this finite set generates $H$.
\end{proof}

\begin{proposition}[Dimension one]
\label{prop:dimension-one}
For every field $k$,
\[
\operatorname{GA}_1(k)=\operatorname{TA}_1(k)=\operatorname{Aff}_1(k).
\]
If $k$ is finite, this group is finite.
\end{proposition}

\begin{proof}
Let $F\in k[x]$ have a polynomial inverse $G\in k[x]$. Both polynomials are nonconstant, and
\[
1=\deg x=\deg(F\circ G)=(\deg F)(\deg G).
\]
Thus $\deg F=1$, so $F(x)=ax+b$ with $a\in k^\times$. Over a finite field there are only finitely many such maps.
\end{proof}

\begin{remark}[The plane case]
The argument below applies directly for $n=2$ and does not use a structural decomposition of the plane automorphism group. In particular, the non-finite-generation statement in dimension two is obtained from the same Cartier-character quotient as in higher dimensions.
\end{remark}

\section{Polynomial de Rham cohomology and relative Cartier}
\label{sec:cartier}

From now on,
\[
k=\Fq,
\qquad q=p^a,
\qquad R=k[x_1,\ldots,x_n],
\qquad n\ge2.
\]
Let
\[
S=\Spec\Fp,
\qquad X=\Spec R.
\]
We use the polynomial de Rham complex relative to the prime field,
\[
\Omega_R^\bullet=\Omega_{R/\Fp}^\bullet.
\]
The extension $k/\Fp$ is finite separable and hence finite étale, so $\Omega^1_{k/\Fp}=0$. The transitivity sequence for Kähler differentials therefore identifies
\[
\Omega_{R/\Fp}^\bullet\cong\Omega_{R/k}^\bullet.
\]
Write
\[
H^m_{\mathrm{dR}}(R)=H^m(\Omega_R^\bullet).
\]

For $I=\{i_1<\cdots<i_m\}\subseteq\{1,\ldots,n\}$, put
\[
dx_I=dx_{i_1}\wedge\cdots\wedge dx_{i_m},
\qquad
\rho_I=
\left(\prod_{i\in I}x_i^{p-1}\right)dx_I.
\]
Since $k$ is perfect,
\[
R^p=\{f^p:f\in R\}=k[x_1^p,\ldots,x_n^p].
\]
Here $R^p$ is ring notation and is unrelated to the group notation $G^{[p]}$.

\begin{lemma}[Polynomial de Rham cohomology]
\label{lem:derham-basis}
For every $0\le m\le n$,
\[
H^m_{\mathrm{dR}}(R)
=
\bigoplus_{\substack{I\subseteq\{1,\ldots,n\}\\|I|=m}}
R^p[\rho_I].
\]
Equivalently, a $k$-basis is given by the classes of $x^\alpha dx_I$ for which
\[
\alpha_i\equiv p-1\pmod p\quad(i\in I),
\qquad
\alpha_i\equiv0\pmod p\quad(i\notin I).
\]
\end{lemma}

\begin{proof}
For one variable the de Rham complex is
\[
0\longrightarrow k[x]
\xrightarrow{\ d\ }
k[x]dx
\longrightarrow0.
\]
Since $d(x^r)=rx^{r-1}dx$, its cohomology is
\[
H^0=k[x^p],
\qquad
H^1=k[x^p]x^{p-1}dx.
\]
Indeed, $x^sdx$ is a derivative unless $s\equiv p-1\pmod p$, and the kernel in degree zero is $k[x^p]$.

The $n$-variable de Rham complex is the tensor product over $k$ of the $n$ one-variable complexes. Since $k$ is a field, the Künneth formula has no Tor term. Choosing degree-zero or degree-one cohomology in each tensor factor gives the asserted decomposition.
\end{proof}

Assign weight one to every $x_i$ and every $dx_i$. The differential preserves weight, so the de Rham complex is a direct sum of homogeneous weight subcomplexes,
\[
\Omega_R^\bullet
=
\bigoplus_{D\ge0}\Omega_R^\bullet[D],
\]
and de Rham cohomology inherits the corresponding direct-sum weight grading.

Set
\[
\omega=dx_1\wedge\cdots\wedge dx_n
\]
and, for $1\le j\le n$,
\[
\omega_j=(-1)^{j-1}
 dx_1\wedge\cdots\wedge\widehat{dx_j}\wedge\cdots\wedge dx_n,
\qquad
dx_j\wedge\omega_j=\omega.
\]
Put
\[
\rho_j=
\left(\prod_{i\ne j}x_i^{p-1}\right)\omega_j.
\]
Then \cref{lem:derham-basis} gives
\begin{equation}
H^{n-1}_{\mathrm{dR}}(R)
=
\bigoplus_{j=1}^nR^p[\rho_j].
\end{equation}
Its lowest nonzero weight is
\begin{equation}
w_0=p(n-1),
\end{equation}
and the corresponding subspace is
\begin{equation}
W_0=\spanop_k\{[\rho_1],\ldots,[\rho_n]\}.
\end{equation}
Every other nonzero weight is at least $w_0+p$. Let
\[
\pr_0:H^{n-1}_{\mathrm{dR}}(R)\longrightarrow W_0
\]
be the projection associated with the weight decomposition.

\subsection{Relative Cartier and its naturality}

The morphism $\Spec k\to S$ is finite étale and hence smooth, and $X=\A_k^n\to\Spec k$ is smooth. Their composite $X\to S$ is therefore smooth. Let
\[
X^{(p)}=X\times_{S,\operatorname{Fr}_S}S
\]
and let
\[
\operatorname{Fr}_{X/S}:X\longrightarrow X^{(p)}
\]
be relative Frobenius; see \cite[Tag~0CC6, Definition~33.36.4]{StacksProject}. Functoriality of relative Frobenius is \cite[Tag~0CC6, Lemma~33.36.5]{StacksProject}. Since the Frobenius of $S=\Spec\Fp$ is the identity, the projection $X^{(p)}\to X$ is a canonical isomorphism of $S$-schemes. On coordinate rings this is
\[
R^{(p)}=R\otimes_{\Fp,\operatorname{Fr}_{\Fp}}\Fp
=R\otimes_{\Fp,\id}\Fp
\cong R.
\]

The relative Cartier theorem for the smooth morphism $X\to S$ gives an isomorphism of graded $\mathcal O_{X^{(p)}}$-algebras
\[
C^{-1}_{X/S}:
\Omega^m_{X^{(p)}/S}
\xrightarrow{\ \sim\ }
\mathcal H^m\!\left(
\operatorname{Fr}_{X/S,*}\Omega^\bullet_{X/S}
\right);
\]
see \cite[Theorem~7.2]{Katz1970}. Since $X^{(p)}$ is affine, taking global sections identifies the right-hand side with the cohomology of the complex of global polynomial forms. Under the canonical identification $X^{(p)}\cong X$, we obtain an additive multiplicative isomorphism
\[
C^{-1}:\Omega_R^m\xrightarrow{\ \sim\ }H^m_{\mathrm{dR}}(R).
\]
As a homomorphism of graded algebras, inverse Cartier is characterized on functions and their differentials by
\begin{equation}
C^{-1}(g)=[g^p],
\qquad
C^{-1}(dg)=[g^{p-1}dg]
\qquad(g\in R).
\label{eq:Cartier-generators}
\end{equation}
Equivalently, for a coordinate monomial form one has
\begin{equation}
C^{-1}(h\,dx_I)
=
\left[
 h^p\left(\prod_{i\in I}x_i^{p-1}\right)dx_I
\right].
\label{eq:Cartier-inverse}
\end{equation}
We denote its inverse by
\[
C:H^m_{\mathrm{dR}}(R)\longrightarrow\Omega_R^m.
\]

\begin{lemma}[Coefficient semilinearity]
\label{lem:Cartier-semilinearity}
For $c\in k$ and $\alpha\in H^m_{\mathrm{dR}}(R)$,
\[
C(c\alpha)=c^{1/p}C(\alpha),
\]
where $c^{1/p}$ is the unique $p$th root of $c$ in $k$.
\end{lemma}

\begin{proof}
Choose $b\in k$ with $b^p=c$, and write $\alpha=C^{-1}(\eta)$. The coordinate formula \eqref{eq:Cartier-inverse} gives
\[
c\alpha=b^pC^{-1}(\eta)=C^{-1}(b\eta).
\]
Applying $C$ yields
\[
C(c\alpha)=b\eta=c^{1/p}C(\alpha).
\]
\end{proof}

In degree $n-1$, \eqref{eq:Cartier-inverse} gives
\begin{equation}
C([h^p\rho_j])=h\omega_j.
\label{eq:Cartier-degree-n-minus-one}
\end{equation}

\begin{proposition}[Cartier naturality]
\label{prop:Cartier-naturality}
Let $F:X\to X$ be a polynomial $k$-endomorphism. For every de Rham class $\alpha$,
\[
C(F^*\alpha)=F^*C(\alpha).
\]
Here the Frobenius twist is taken relative to
$S=\operatorname{Spec}\mathbf F_p$, not relative to
$\operatorname{Spec}k$.
\end{proposition}

\begin{proof}
Before identifying the Frobenius twists, functoriality of relative Frobenius (the cited Stacks lemma) gives a commutative square
\[
\begin{CD}
X @>{F}>> X\\
@V{\operatorname{Fr}_{X/S}}VV
@VV{\operatorname{Fr}_{X/S}}V\\
X^{(p)} @>{F^{(p)}}>> X^{(p)}.
\end{CD}
\]
We verify the corresponding naturality of inverse Cartier from its coordinate formulas.

For $g\in R$, the generator identities \eqref{eq:Cartier-generators} give
\[
F^*C^{-1}(g)
=F^*[g^p]
=[(F^*g)^p]
=C^{-1}(F^{(p)*}g),
\]
and
\[
\begin{aligned}
F^*C^{-1}(dg)
&=F^*[g^{p-1}dg]\\
&=[(F^*g)^{p-1}d(F^*g)]\\
&=C^{-1}\!\left(d(F^{(p)*}g)\right).
\end{aligned}
\]
Since the algebra of differential forms is generated by functions and their differentials, additivity and multiplicativity give
\[
F^*C^{-1}(\eta)=C^{-1}(F^{(p)*}\eta)
\]
for every differential form $\eta$.

Let $\eta=C(\alpha)$. Then
\[
F^*\alpha
=F^*C^{-1}(\eta)
=C^{-1}(F^{(p)*}\eta).
\]
Applying $C$ gives
\[
C(F^*\alpha)=F^{(p)*}C(\alpha).
\]

Because the twist is relative to $S=\operatorname{Spec}\mathbf F_p$,
whose Frobenius is the identity, the canonical identification
$X^{(p)}\cong X$ is the relative $S$-scheme identification. In particular,
this is not the Frobenius twist over $\operatorname{Spec}k$; no Frobenius is
applied to coefficients in $k$. Under this identification, the coordinate
ring of the twist is $R\otimes_{\Fp,\id}\Fp\cong R$, and $F^{(p)}$ has exactly
the same coordinate polynomials and coefficients as $F$. Hence
$F^{(p)*}=F^*$, proving the asserted formula.
\end{proof}

The operator $C$ does not generally send a cohomology class to a closed form. Thus it cannot be iterated on all de Rham classes.

\begin{example}[Failure of unrestricted Cartier iteration]
\label{ex:Cartier-not-iterable}
The class $[x_1^p\rho_1]$ is well defined, but
\[
C([x_1^p\rho_1])=x_1\omega_1
\]
and
\[
d(x_1\omega_1)=dx_1\wedge\omega_1=\omega\ne0.
\]
\end{example}

Whenever $C(\alpha)$ is closed, define the partial operation
\begin{equation}
\mathcal T(\alpha)=[C(\alpha)]\in H^{n-1}_{\mathrm{dR}}(R).
\end{equation}

The same relative-Cartier decomposition underlies Ma's completed formal
flux construction \cite[Sections~3.1--3.2]{MaCartierFrattini2026}. The exact
comparison with the polynomial tame characters is deferred to
\cref{prop:Ma-comparison}.

\section{Flux and tame Cartier-admissibility}
\label{sec:flux}

Set
\[
\lambda=x_1\omega_1=x_1\,dx_2\wedge\cdots\wedge dx_n.
\]
Then $d\lambda=\omega$. The corresponding completed formal
pullback-difference construction appears in
\cite[Section~3.2 and Proposition~3.3]{MaCartierFrattini2026}; here we work with polynomial
forms and then prove repeated descent on the tame subgroup.

\begin{definition}[Flux]
For $F\in\operatorname{STA}_n(k)$, define
\[
\Phi_0(F)=[F^*\lambda-\lambda]
\in H^{n-1}_{\mathrm{dR}}(R).
\]
The form is closed because
\[
d(F^*\lambda-\lambda)
=F^*\omega-\omega
=(\det JF-1)\omega=0.
\]
\end{definition}

\begin{lemma}[Flux cocycle]
\label{lem:flux-cocycle}
For $F,G\in\operatorname{STA}_n(k)$,
\begin{equation}
\Phi_0(FG)=G^*\Phi_0(F)+\Phi_0(G).
\end{equation}
\end{lemma}

\begin{proof}
Using $(FG)^*=G^*F^*$, one obtains
\[
\begin{aligned}
\Phi_0(FG)
&=[G^*F^*\lambda-\lambda]\\
&=G^*[F^*\lambda-\lambda]+[G^*\lambda-\lambda].
\end{aligned}
\]
\end{proof}

\subsection{Flux on tame generators}

\begin{lemma}[Affine flux and the exceptional plane]
\label{lem:affine-flux}
Let $A\in\operatorname{SAff}_n(k)$.
\begin{enumerate}[label=\textup{(\roman*)}]
\item If $(n,p)\ne(2,2)$, then
\[
\Phi_0(A)=0.
\]
\item If $(n,p)=(2,2)$, then
\[
\Phi_0(A)\in W_0
\qquad\text{and}\qquad
\mathcal T(\Phi_0(A))=0.
\]
\end{enumerate}
\end{lemma}

\begin{proof}
Since $A$ is affine, every coefficient of $A^*\lambda-\lambda$ has polynomial degree at most one, while its differential part has degree $n-1$. Thus every homogeneous term has weight at most $n$.

The lowest nonzero weight of $H^{n-1}_{\mathrm{dR}}(R)$ is
\[
w_0=p(n-1).
\]
For $n,p\ge2$,
\[
p(n-1)-n=(p-1)n-p.
\]
If $p=2$, this equals $n-2$ and vanishes only for $n=2$. If $p\ge3$, then
\[
(p-1)n-p\ge2(p-1)-p=p-2>0.
\]
Hence $n<w_0$ unless $(n,p)=(2,2)$. In the nonexceptional cases there is no nonzero de Rham cohomology in the weights occurring in $A^*\lambda-\lambda$, so $\Phi_0(A)=0$.

Suppose now that $n=p=2$. Decompose the closed one-form $A^*\lambda-\lambda$ into its weight-one and weight-two parts,
\[
A^*\lambda-\lambda=\theta_1+\theta_2.
\]
Since $d$ preserves weight, both $\theta_1$ and $\theta_2$ are closed. The group $H^1_{\mathrm{dR}}(R)$ has no nonzero weight-one component, so $[\theta_1]=0$. Its weight-two component is
\[
W_0=\spanop_k\{[x_2dx_2],[x_1dx_1]\},
\]
and therefore $\Phi_0(A)\in W_0$.

Cartier sends these basis classes to
\[
C([x_2dx_2])=dx_2,
\qquad
C([x_1dx_1])=dx_1,
\]
with inverse Frobenius applied to their coefficients. Both constant one-forms are exact. Hence $\mathcal T(\Phi_0(A))=0$.
\end{proof}

\begin{remark}[The exceptional flux can be nonzero]
When $n=p=2$, take
\[
A=E_1(x_2)=(x_1+x_2,x_2).
\]
Then
\[
A^*\lambda-\lambda=x_2dx_2,
\]
and $[x_2dx_2]\ne0$ by \cref{lem:derham-basis}. Thus the exceptional clause in \cref{lem:affine-flux} cannot be replaced by affine-flux vanishing.
\end{remark}

\begin{lemma}[Elementary flux]
\label{lem:elementary-flux}
Let $E_j(f)$ be elementary, with $f$ independent of $x_j$. Then
\begin{equation}
\Phi_0(E_j(f))=[f\omega_j].
\label{eq:elementary-flux}
\end{equation}
\end{lemma}

\begin{proof}
For $j=1$,
\[
E_1(f)^*\lambda-\lambda=f\omega_1.
\]
Assume $j\ge2$, and put
\[
\eta_j=dx_2\wedge\cdots\wedge\widehat{dx_j}\wedge\cdots\wedge dx_n.
\]
The form $df$ occurs in the $(j-1)$st position of $dx_2\wedge\cdots\wedge dx_n$, so
\[
E_j(f)^*\lambda-\lambda
=(-1)^{j-2}x_1df\wedge\eta_j.
\]
Moreover,
\[
d(x_1f\eta_j)=f\,dx_1\wedge\eta_j+x_1df\wedge\eta_j.
\]
Hence
\[
[E_j(f)^*\lambda-\lambda]
=(-1)^{j-1}[f\,dx_1\wedge\eta_j].
\]
By the definition of $\omega_j$,
\[
dx_1\wedge\eta_j=(-1)^{j-1}\omega_j.
\]
The two signs cancel, giving \eqref{eq:elementary-flux}.
\end{proof}

For
\[
f=\sum_{\alpha_j=0}c_\alpha x^\alpha
\in k[x_1,\ldots,\widehat{x_j},\ldots,x_n],
\]
define the $j$th Cartier residue by
\begin{equation}
\mathcal R_j(f)=
\sum_{\substack{\alpha_j=0\\
\alpha_i\equiv p-1\ (\mathrm{mod}\ p),\ i\ne j}}
c_\alpha^{1/p}
\prod_{i\ne j}x_i^{(\alpha_i-(p-1))/p}.
\end{equation}

\begin{lemma}[Elementary Cartier descent]
\label{lem:elementary-Cartier}
For $f$ independent of $x_j$,
\begin{equation}
C([f\omega_j])=\mathcal R_j(f)\omega_j.
\label{eq:elementary-Cartier}
\end{equation}
The form on the right is closed. Repeated application of $\mathcal T$ to $[f\omega_j]$ is defined and eventually gives zero.
\end{lemma}

\begin{proof}
By \cref{lem:derham-basis}, the only monomials of $f\omega_j$ that survive in cohomology are those whose exponent in every $x_i$, $i\ne j$, is congruent to $p-1$ modulo $p$. Thus
\[
[f\omega_j]=[\mathcal R_j(f)^p\rho_j].
\]
Equation \eqref{eq:Cartier-degree-n-minus-one} gives \eqref{eq:elementary-Cartier}.

The polynomial $\mathcal R_j(f)$ is independent of $x_j$, so
\[
d(\mathcal R_j(f)\omega_j)
=
\frac{\partial\mathcal R_j(f)}{\partial x_j}
 dx_j\wedge\omega_j
=0.
\]
Thus another Cartier step is defined and has the same form.

Suppose $\mathcal R_j(f)\ne0$, and put $d=\deg_{\mathrm{tot}}f$. Every surviving monomial has original degree
\[
p\,\deg(\text{residue monomial})+(n-1)(p-1).
\]
Consequently,
\[
\deg\mathcal R_j(f)
\le
\frac{d-(n-1)(p-1)}p
<d.
\]
The total degree therefore decreases at every nonzero step. If an iterate of the residue polynomial is constant, the corresponding form is a constant multiple of $\omega_j$, whose de Rham class is zero. Indeed, choose $i\ne j$ and a constant $(n-2)$-form $\eta$ satisfying $dx_i\wedge\eta=\omega_j$; then
\[
\omega_j=d(x_i\eta).
\]
Hence the iteration terminates.
\end{proof}

\subsection{Admissibility under tame composition}

\begin{definition}[Cartier-admissibility]
An element $F\in\operatorname{STA}_n(k)$ is \emph{Cartier-admissible} if the sequence
\[
\Phi_{r+1}(F)=\mathcal T(\Phi_r(F)),
\qquad r\ge0,
\]
is defined for every $r$ and is eventually zero.
\end{definition}

\begin{lemma}[Closure of admissibility under products]
\label{lem:admissible-products}
Let $F,G\in\operatorname{STA}_n(k)$ be Cartier-admissible. Then $FG$ is Cartier-admissible, and for every $r\ge0$,
\begin{equation}
\Phi_r(FG)=G^*\Phi_r(F)+\Phi_r(G).
\label{eq:iterated-cocycle}
\end{equation}
If
\[
\Phi_r(F)=0\quad(r>N_F),
\qquad
\Phi_r(G)=0\quad(r>N_G),
\]
then
\[
\Phi_r(FG)=0
\qquad
(r>\max\{N_F,N_G\}).
\]
\end{lemma}

\begin{proof}
We prove simultaneously, by induction on $r$, that $\Phi_r(FG)$ is defined and that \eqref{eq:iterated-cocycle} holds.

For $r=0$, this is \cref{lem:flux-cocycle}. Assume the assertion at level $r$. Applying Cartier, using additivity and \cref{prop:Cartier-naturality}, gives
\[
\begin{aligned}
C(\Phi_r(FG))
&=C\bigl(G^*\Phi_r(F)+\Phi_r(G)\bigr)\\
&=G^*C(\Phi_r(F))+C(\Phi_r(G)).
\end{aligned}
\]
Since $F$ and $G$ are admissible, the two Cartier images on the right are closed. Pullback preserves closed forms, so their sum is closed. Hence $\Phi_{r+1}(FG)$ is defined. Taking de Rham classes yields
\[
\Phi_{r+1}(FG)
=G^*\Phi_{r+1}(F)+\Phi_{r+1}(G).
\]

If $r>\max\{N_F,N_G\}$, both terms on the right of \eqref{eq:iterated-cocycle} vanish. This proves the termination assertion.
\end{proof}

\begin{theorem}[Tame Cartier-admissibility]
\label{thm:tame-admissibility}
Every element of $\operatorname{STA}_n(k)$ is Cartier-admissible. For every $F,G\in\operatorname{STA}_n(k)$ and every $r\ge0$,
\[
\Phi_r(FG)=G^*\Phi_r(F)+\Phi_r(G).
\]
\end{theorem}

\begin{proof}
Special affine automorphisms are Cartier-admissible by \cref{lem:affine-flux}, and their inverses are again special affine. Elementary automorphisms are Cartier-admissible by \cref{lem:elementary-flux,lem:elementary-Cartier}; moreover,
\[
E_j(f)^{-1}=E_j(-f).
\]
By \cref{prop:special-generation}, every element of $\operatorname{STA}_n(k)$ is a finite product of special affine and elementary automorphisms and their inverses. Repeated application of \cref{lem:admissible-products} proves both admissibility and the iterated crossed-cocycle identity.
\end{proof}

The tame hypothesis enters precisely through \cref{prop:special-generation}: admissibility is verified on the special affine and elementary generators and then propagated through finite tame words.

\section{Lowest-weight Cartier characters and finite jets}
\label{sec:characters}

The maps $\Phi_r$ are crossed homomorphisms. On $U_n(k)$, their pullback term disappears after projection to the lowest nonzero weight.

For $d\ge0$, define the split descending weight filtrations
\[
\mathcal F^{\ge d}\Omega_R^\bullet
=
\bigoplus_{e\ge d}\Omega_R^\bullet[e],
\qquad
\mathcal F^{\ge d}H^m_{\mathrm{dR}}(R)
=
\bigoplus_{e\ge d}H^m(\Omega_R^\bullet[e]).
\]
Their associated graded pieces are the homogeneous weight-$d$ summands.

\begin{lemma}[Tangent pullback on the weight filtration]
\label{lem:tangent-weight}
Let $F\in U_n(k)$. For every $m\ge0$, pullback by $F$ preserves the weight filtration on $H^m_{\mathrm{dR}}(R)$ and induces the identity on its associated graded:
\[
\operatorname{gr}_d(F^*)
=
\id_{\operatorname{gr}_dH^m_{\mathrm{dR}}(R)}
\qquad(d\ge0).
\]
In particular, for $\alpha\in H^{n-1}_{\mathrm{dR}}(R)$,
\begin{equation}
\pr_0(F^*\alpha)=\pr_0(\alpha).
\label{eq:lowest-weight-invariance}
\end{equation}
\end{lemma}

\begin{proof}
Let
\[
\mathfrak m=(x_1,\ldots,x_n).
\]
Since $F(0)=0$ and $JF(0)=I_n$, one may write
\[
F_i=x_i+h_i,
\qquad
h_i\in\mathfrak m^2.
\]
Consequently,
\[
dF_i=dx_i+dh_i.
\]

Let $\theta=x^\alpha dx_I$ be a homogeneous monomial form of weight $d$. In the expansion of $F^*\theta$, choosing $x_i$ from every occurrence of $F_i$ and $dx_i$ from every occurrence of $dF_i$ produces exactly $\theta$.

Every other nonzero term contains a homogeneous component of some $h_i$ or $dh_i$. If that component comes from a homogeneous polynomial of degree $e\ge2$, then both the polynomial component and its differential have total weight $e$. It replaces a weight-one factor, so the total weight increases by at least one. Therefore
\[
F^*\theta-\theta
\in
\bigoplus_{e\ge d+1}\Omega_R^\bullet[e].
\]

Thus $F^*$ preserves the weight filtration and induces the identity on the associated graded de Rham complex. Since
\[
\Omega_R^\bullet
=
\bigoplus_{d\ge0}\Omega_R^\bullet[d]
\]
is a direct sum of subcomplexes, one has
\[
H^m_{\mathrm{dR}}(R)
=
\bigoplus_{d\ge0}H^m(\Omega_R^\bullet[d]).
\]
Hence passage to cohomology preserves the split weight grading, and $F^*$ induces the identity on its associated graded. The space $W_0$ is the lowest nonzero weight component of $H^{n-1}_{\mathrm{dR}}(R)$, so \eqref{eq:lowest-weight-invariance} follows.
\end{proof}

For $1\le j\le n$, let
\[
\ell_j:W_0\longrightarrow k
\]
be determined by
\[
\ell_j([\rho_i])=\delta_{ij}.
\]
For $r\ge0$, define
\begin{equation}
\chi_{r,j}:U_n(k)\longrightarrow(k,+),
\qquad
\chi_{r,j}(F)=
\ell_j\bigl(\pr_0\Phi_r(F)\bigr).
\end{equation}

\begin{theorem}[Iterated lowest-weight Cartier characters]
\label{thm:characters}
Each $\chi_{r,j}$ is a group homomorphism. For every fixed $F\in U_n(k)$, one has
\[
\chi_{r,j}(F)=0
\]
for every $j$ and all sufficiently large $r$.
\end{theorem}

\begin{proof}
For $F,G\in U_n(k)$, \cref{thm:tame-admissibility,lem:tangent-weight} give
\[
\begin{aligned}
\chi_{r,j}(FG)
&=\ell_j\!\left(
\pr_0\bigl(G^*\Phi_r(F)+\Phi_r(G)\bigr)
\right)\\
&=\ell_j(\pr_0\Phi_r(F))
 +\ell_j(\pr_0\Phi_r(G))\\
&=\chi_{r,j}(F)+\chi_{r,j}(G).
\end{aligned}
\]
Thus $\chi_{r,j}$ is a homomorphism.

By Cartier-admissibility, for every $F$ there is an integer $N$ such that $\Phi_r(F)=0$ for $r>N$. All coordinate characters therefore vanish for $r>N$.
\end{proof}

In the common prime-field range, the vector of zeroth characters is the
restriction of Ma's constant Cartier coordinate. The new subgroup-sensitive
phenomenon begins at the positive levels $r\ge1$; see
\cref{prop:Ma-comparison}. Elementwise finite support is part of the same
iterated-character theorem and makes the joint direct-sum map well defined.

\subsection{Finite-jet dependence}

For $N\ge1$, let
\[
\operatorname{Jet}_N
=
\Aut_k\bigl(\Spec(R/\mathfrak m^{N+1})\bigr),
\]
and let
\[
\jet_N:U_n(k)\longrightarrow\operatorname{Jet}_N
\]
be the restriction homomorphism. Since $k$ and $R/\mathfrak m^{N+1}$ are finite, $\operatorname{Jet}_N$ and the image of $\jet_N$ are finite.

\begin{proposition}[Finite-jet dependence]
\label{prop:finite-jet}
For $r\ge0$, put
\begin{equation}
N_r=(n-1)(p^{r+1}-1).
\end{equation}
Then, for every $1\le j\le n$,
\[
\ker(\jet_{N_r})\subseteq\ker\chi_{r,j}.
\]
Equivalently, $\chi_{r,j}$ factors uniquely through the finite image of $\jet_{N_r}$.
\end{proposition}

\begin{proof}
The weight decomposition of the de Rham complex is a direct sum of subcomplexes. Hence closedness, exactness, and cohomology all decompose weight by weight.

If $h$ is homogeneous and $h\omega_j$ has weight $D$, then
\[
C^{-1}(h\omega_j)=[h^p\rho_j]
\]
has weight $pD$. Thus Cartier sends a homogeneous class of weight $pD$ to a homogeneous form of weight $D$. Whenever that form is closed, passage to its de Rham class preserves its weight.

Since the lowest nonzero weight in $H^{n-1}_{\mathrm{dR}}(R)$ is
\[
w_0=p(n-1),
\]
the $W_0$-component of $\Phi_r(F)$ can arise only from the weight
\[
D_r=p^rw_0=p^{r+1}(n-1)
\]
component of $\Phi_0(F)$.

Now
\[
F^*\lambda
=F_1\,dF_2\wedge\cdots\wedge dF_n.
\]
Because $F(0)=0$, every homogeneous coordinate component selected from any $F_i$ has degree at least one. A homogeneous component of $dF_i$ arising from a coordinate component of degree $d_i$ has total weight $d_i$. Therefore a nonzero homogeneous term of $F^*\lambda$ obtained from coordinate degrees $d_1,\ldots,d_n$ has weight
\[
d_1+\cdots+d_n.
\]

If this weight is $D_r$, then $d_i\ge1$ for all $i$, and hence
\[
d_i
\le D_r-(n-1)
=(n-1)(p^{r+1}-1)
=N_r.
\]
It follows that the weight-$D_r$ component of $F^*\lambda-\lambda$ depends only on the $N_r$-jet of $F$. The same is therefore true of $\pr_0\Phi_r(F)$ and of every $\chi_{r,j}(F)$.

If $F\in\ker(\jet_{N_r})$, then $F$ and the identity have the same $N_r$-jet. Consequently,
\[
\chi_{r,j}(F)=\chi_{r,j}(\id)=0.
\]
The asserted factorization follows.
\end{proof}

\begin{remark}[The zeroth layer in characteristic two]
If $(n,p)=(2,2)$, then
\[
N_0=(2-1)(2-1)=1.
\]
Every element of $U_2(k)$ has the same first jet as the identity. By \cref{prop:finite-jet},
\[
\chi_{0,j}=0
\qquad(j=1,2).
\]
Moreover, the formal test shear with $s=0$ introduced below would have a nonzero linear term and would not belong to $U_2(k)$. This accounts for the definition
\[
S_{2,2}=\N_{\ge1}.
\]
\end{remark}

\section{Test shears, Cartier quotients, and finite generation}
\label{sec:quotient}

For $m\ge1$ and $1\le j\le n$, put
\begin{equation}
\rho_j^{(m)}=
\left(\prod_{i\ne j}x_i^{p^m-1}\right)\omega_j.
\end{equation}
Thus $\rho_j^{(1)}=\rho_j$, and \eqref{eq:Cartier-degree-n-minus-one} gives
\begin{equation}
C([\rho_j^{(m)}])=\rho_j^{(m-1)}\quad(m\ge2),
\qquad
C([\rho_j])=\omega_j.
\label{eq:rho-Cartier-chain}
\end{equation}
The form $\omega_j$ is exact. Indeed, choose $i\ne j$ and a constant $(n-2)$-form $\eta$ with $dx_i\wedge\eta=\omega_j$; then
\[
\omega_j=d(x_i\eta).
\]

For $s\in S_{n,p}$, $a\in k$, and $1\le j\le n$, define
\begin{equation}
\label{eq:test-shears}
v_{s,j}(a)=
E_j\!\left(
 a^{p^s}\prod_{i\ne j}x_i^{p^{s+1}-1}
\right).
\end{equation}
The one-coordinate inverse-Cartier monomial pattern is shared with Ma's
formal representative construction \cite[Lemma~3.6]{MaCartierFrattini2026}.
The contribution used here is the Frobenius-normalized infinite family
\eqref{eq:test-shears} and its simultaneous cross-level delta evaluation.

\begin{lemma}[Membership of the test shears]
For every $s\in S_{n,p}$, $a\in k$, and $1\le j\le n$,
\[
v_{s,j}(a)\in U_n(k).
\]
\end{lemma}

\begin{proof}
The added monomial has degree
\[
(n-1)(p^{s+1}-1).
\]
This degree equals one precisely when $(n,p,s)=(2,2,0)$, which is excluded by the definition of $S_{n,p}$. In every allowed case the added term has degree at least two. Thus $v_{s,j}(a)$ fixes the origin and has identity linear part.
\end{proof}

The first two examples display the residue mechanism and the exceptional index convention.

\begin{example}[The zeroth layer in odd characteristic]
Let $n=2$ and $p=3$. Then
\[
v_{0,1}(a)=E_1(ax_2^2).
\]
Its flux is
\[
\Phi_0(v_{0,1}(a))=a[x_2^2dx_2]=a[\rho_1],
\]
so
\[
\chi_{0,1}(v_{0,1}(a))=a.
\]
\end{example}

\begin{example}[The first nonzero layer when $n=p=2$]
Let $n=p=2$. Then $0\notin S_{2,2}$, while
\[
v_{1,1}(a)=E_1(a^2x_2^3)
\]
belongs to $U_2(k)$. One has
\[
\Phi_0(v_{1,1}(a))=a^2[x_2^3dx_2]
\]
and
\[
\Phi_1(v_{1,1}(a))=a[x_2dx_2].
\]
Hence
\[
\chi_{1,1}(v_{1,1}(a))=a.
\]
\end{example}

\begin{proposition}[Kronecker-delta evaluation]
\label{prop:delta-evaluation}
For $r\ge0$, $s\in S_{n,p}$, $1\le i,j\le n$, and $a\in k$,
\begin{equation}
\chi_{r,i}(v_{s,j}(a))
=
\delta_{rs}\delta_{ij}a.
\end{equation}
\end{proposition}

\begin{proof}
By \cref{lem:elementary-flux},
\[
\Phi_0(v_{s,j}(a))
=a^{p^s}[\rho_j^{(s+1)}].
\]
Using \cref{lem:Cartier-semilinearity} and \eqref{eq:rho-Cartier-chain}, one obtains
\[
\Phi_r(v_{s,j}(a))
=
\begin{cases}
 a^{p^{s-r}}[\rho_j^{(s+1-r)}],&0\le r\le s,\\
 0,&r>s.
\end{cases}
\]

If $r<s$, the class $[\rho_j^{(s+1-r)}]$ has weight
\[
(n-1)p^{s+1-r}>p(n-1)=w_0,
\]
so its projection to $W_0$ is zero. If $r=s$, then
\[
\Phi_s(v_{s,j}(a))=a[\rho_j],
\]
and therefore
\[
\chi_{s,i}(v_{s,j}(a))
=\ell_i(a[\rho_j])
=\delta_{ij}a.
\]
For $r>s$, the iterated class is zero because the next Cartier image is a scalar multiple of the exact form $\omega_j$.
\end{proof}

\begin{corollary}[Optimality of the jet order]
\label{cor:optimal-jet}
Let $r\in S_{n,p}$ and $1\le j\le n$. Then $N_r$ is the least integer $N$ for which $\chi_{r,j}$ factors through $\jet_N$.
\end{corollary}

\begin{proof}
By \cref{prop:finite-jet}, the character factors through the $N_r$-jet. The test element $v_{r,j}(1)$ differs from the identity by a single coordinate term of degree exactly
\[
(n-1)(p^{r+1}-1)=N_r.
\]
Hence
\[
v_{r,j}(1)\in\ker(\jet_{N_r-1}).
\]
On the other hand, \cref{prop:delta-evaluation} gives
\[
\chi_{r,j}(v_{r,j}(1))=1.
\]
Thus $\chi_{r,j}$ does not factor through the $(N_r-1)$-jet. If $N<N_r$, truncation induces a homomorphism
\[
\tau_{N_r-1,N}:\operatorname{im}(\jet_{N_r-1})
\longrightarrow\operatorname{im}(\jet_N)
\]
with
\[
\jet_N=\tau_{N_r-1,N}\circ\jet_{N_r-1}.
\]
Consequently, factorization through any $N$-jet with $N<N_r$ would imply factorization through the $(N_r-1)$-jet, which has just been excluded. Hence $N_r$ is minimal.
\end{proof}

\subsection{Comparison with ambient formal Cartier coordinates}
\label{subsec:Ma-comparison}

The following proposition records the exact common input and the point at
which the polynomial tame-subgroup construction diverges from Ma's ambient
formal Frattini calculation. It is not used in the proof of
\cref{thm:main}; it depends on the character and test-shear results already
proved above.

Assume in this subsection that
\[
k=\Fp,\qquad n\ge3,\qquad p\ge5.
\]
Let $\widehat R=\Fp[[x_1,\ldots,x_n]]$ and
$\widehat{\mathfrak m}=(x_1,\ldots,x_n)$. Let $\widehat G_1$ be the group of
continuous $\Fp$-algebra automorphisms of $\widehat R$ which are tangent to
the identity and preserve
$dx_1\wedge\cdots\wedge dx_n$, and put
\[
\widehat G_M=
\{g\in\widehat G_1:g(x_i)-x_i\in\widehat{\mathfrak m}^{M}
\text{ for every }i\},
\qquad
T_M=\widehat G_1/\widehat G_M.
\]
Polynomial completion maps $U_n(\Fp)$ into $\widehat G_1$; denote its image
in $T_M$ by $U_M$. Notice that our $N$-jet is defined using
$R/\mathfrak m^{N+1}$. Thus $M>N$ guarantees that two elements with the same
image in $T_M$ have the same $N$-jet. This explains the strict cutoff in
part~\textup{(iv)} below.

\begin{proposition}[Comparison with ambient formal Cartier coordinates]
\label{prop:Ma-comparison}
For $r\ge0$, put
\[
q_r=(n-1)(p^r-1),
\qquad
d_t=(p-1)(n-1)+pt.
\]
Then the following statements hold.
\begin{enumerate}[label=\textup{(\roman*)}]
\item One has $N_r=d_{q_r}$.

\item Under the standard Cartier--contraction identification of the
index-zero Cartier quotient with $V^{(1)}$, where $V=\Fp^n$, the vector
character
\[
(\chi_{0,1},\ldots,\chi_{0,n})
\]
agrees with the restriction to $U_n(\Fp)$ of Ma's $Q_0$-coordinate.

\item The test shear $v_{r,j}(a)$ is the one-coordinate inverse-Cartier
representative associated with the multi-index
\[
\alpha_j=0,
\qquad
\alpha_i=p^r-1\quad(i\ne j).
\]

\item Let $r\ge1$ and $M>N_r$. The character induced by $\chi_{r,j}$ on
$U_M$ does not extend to a homomorphism
\[
T_M\longrightarrow(\Fp,+).
\]
\end{enumerate}
\end{proposition}

\begin{proof}
For part~\textup{(i)}, direct calculation gives
\[
\begin{aligned}
d_{q_r}
&=(p-1)(n-1)+p(n-1)(p^r-1)\\
&=(n-1)(p^{r+1}-1)
=N_r.
\end{aligned}
\]

For part~\textup{(ii)}, inverse Cartier sends
\[
[\rho_j]
=
\left[
\left(\prod_{i\ne j}x_i^{p-1}\right)\omega_j
\right]
\]
to $\omega_j$. Under contraction with
$\omega=dx_1\wedge\cdots\wedge dx_n$, the form $\omega_j$ corresponds to
the constant vector $\partial_j$. These are precisely the coordinate
vectors of Ma's index-zero quotient $Q_0=V^{(1)}$; see
\cite[Theorem~5.2]{MaCartierFrattini2026}. Over $\Fp$ the Frobenius
twist does not change the scalar coordinate, so the restricted coordinates
agree.

For part~\textup{(iii)}, Ma's pure Cartier monomial representative
associated with $x^\alpha\partial_j$ has exponent $p\alpha_j$ in $x_j$ and
exponent $p\alpha_i+p-1$ in $x_i$ for $i\ne j$; see
\cite[Lemma~3.6]{MaCartierFrattini2026}. With the displayed choice of
$\alpha$ these exponents are
\[
p\alpha_j=0,
\qquad
p\alpha_i+p-1=p(p^r-1)+p-1=p^{r+1}-1.
\]
Moreover, $a^{p^r}=a$ in $\Fp$. Hence this representative is exactly
$v_{r,j}(a)$.

For part~\textup{(iv)}, we first make the descent to $U_M$ explicit.
By \cref{prop:finite-jet}, $\chi_{r,j}$ factors through the $N_r$-jet, which
is defined modulo $\mathfrak m^{N_r+1}$. If $F,F'\in U_n(\Fp)$ have the same
image in $U_M$, then their formal completions induce the same automorphism
of $\widehat R/\widehat{\mathfrak m}^{M}$. Since $M>N_r$, one has
$M\ge N_r+1$, so $F$ and $F'$ have the same $N_r$-jet. Hence
$\chi_{r,j}(F)=\chi_{r,j}(F')$, and the induced character on $U_M$ is well
defined.

Since $r\ge1$ and $n\ge3$, $p\ge5$,
\[
q_r=(n-1)(p^r-1)\ge2(5-1)=8\ge3.
\]
Ma's Cartier--Frattini coordinate theorem identifies $\Phi(T_M)$ with the
common kernel of the quadratic coordinate and the retained coordinates
$Q_0,Q_1,Q_2$ \cite[Theorem~5.2]{MaCartierFrattini2026}. By
part~\textup{(iii)}, the element $v_{r,j}(1)$ is a pure index-$q_r$
representative. Its first nonidentity coordinate term has physical degree
$d_{q_r}=N_r<M$, so its quadratic coordinate is zero; its Cartier flux has
no component below index $q_r$, and therefore all retained coordinates
$Q_0,Q_1,Q_2$ vanish because $q_r\ge3$. Thus
\[
v_{r,j}(1)\in\Phi(T_M).
\]
Every homomorphism from the finite $p$-group $T_M$ to $(\Fp,+)$ kills
$\Phi(T_M)$. On the other hand, \cref{prop:delta-evaluation} gives
\[
\chi_{r,j}(v_{r,j}(1))=1.
\]
Consequently no homomorphism $T_M\to(\Fp,+)$ can restrict to the character
induced by $\chi_{r,j}$ on $U_M$.
\end{proof}

The proposition identifies a genuine subgroup effect: the same higher
Cartier source that is killed in the ambient formal Frattini quotient
supports a nonzero ordinary character after restriction to the polynomial
tame subgroup.

\begin{proof}[Proof of \cref{thm:main}]
By \cref{thm:characters}, every image family has finite support.

The delta evaluation supplies a preimage of every standard coordinate vector in the target of $\chi_{\mathrm{all}}$. Given a finitely supported target family, multiply the corresponding test elements in any fixed order. Elements associated with different target coordinates need not commute, but $\chi_{\mathrm{all}}$ is a homomorphism into an abelian group. Its value on the product is therefore the sum of the prescribed standard vectors. Hence $\chi_{\mathrm{all}}$ is surjective.

Fix $j$. For a finitely supported family $(a_s)_{s\in S_{n,p}}$, define
\begin{equation}
\sigma_j((a_s)_s)=\prod_s v_{s,j}(a_s).
\end{equation}
The product is finite. All factors commute because they add polynomials independent of $x_j$ to the same coordinate. Moreover,
\[
v_{s,j}(a)v_{s,j}(b)=v_{s,j}(a+b),
\]
since
\[
(a+b)^{p^s}=a^{p^s}+b^{p^s}.
\]
Thus $\sigma_j$ is a homomorphism. By \cref{prop:delta-evaluation},
\[
\chi_{S,j}\circ\sigma_j=\id.
\]

Every homomorphism from $U_n(k)$ to an abelian group of exponent $p$ kills commutators and $p$th powers. Hence $\chi_{\mathrm{all}}$ factors through $Q_n(k)$, giving a surjection
\[
Q_n(k)\twoheadrightarrow
\bigoplus_{s\in S_{n,p}}(k^n,+).
\]
The target has countably infinite $\Fp$-dimension, so
\[
\dim_{\Fp}Q_n(k)\ge\aleph_0.
\]

Since $k$ is finite, the polynomial ring $R$, the group $U_n(k)$, and the quotient $Q_n(k)$ are countable sets. A vector space over a finite field with uncountable dimension is uncountable. Therefore
\[
\dim_{\Fp}Q_n(k)\le\aleph_0,
\]
and equality follows.
\end{proof}

\begin{remark}[A fixed-coordinate section does not exhaust the quotient]
\label{rem:fixed-section-proper}
Fix $j$, and let $V_j\subseteq Q_n(k)$ be the image of the composite of $\sigma_j$ with the quotient map $U_n(k)\to Q_n(k)$. Choose $i\ne j$ and $s\in S_{n,p}$. Every $\chi_{r,i}$ vanishes on $V_j$ by \cref{prop:delta-evaluation}, whereas
\[
\chi_{s,i}(v_{s,i}(1))=1.
\]
Thus the class of $v_{s,i}(1)$ does not lie in $V_j$, and $V_j$ is a proper $\Fp$-subspace of $Q_n(k)$.
\end{remark}

\begin{proof}[Proof of \cref{cor:intro-classification}]
For $n=1$, \cref{prop:dimension-one} shows that
\[
\operatorname{TA}_1(k)=\operatorname{Aff}_1(k),
\]
which is finite.

Let $n\ge2$. By \cref{thm:main}, the group $U_n(k)$ has a quotient which is not finitely generated. Hence $U_n(k)$ is not finitely generated. If $\operatorname{TA}_n(k)$ were finitely generated, then \cref{prop:finite-index,lem:Schreier} would imply that its finite-index subgroup $U_n(k)$ was finitely generated, a contradiction.
\end{proof}

\begin{corollary}[The modified finite-field Derksen subgroup is proper]
Let $k=\Fq$ and $n\ge3$. Then
\[
\DAMW_n(k)\subsetneq\operatorname{TA}_n(k).
\]
Nevertheless, for every $m\ge1$, the two groups have the same permutation image on $\mathbf F_{q^m}^n$.
\end{corollary}

\begin{proof}
The group $\DAMW_n(k)$ is generated by the finite group $\operatorname{Aff}_n(k)$ and one additional automorphism, so it is finitely generated. It cannot equal the non-finitely-generated group $\operatorname{TA}_n(k)$. Equality of the finite-extension permutation images is \cite[Theorem~3.1]{MaubachWillems2011}.
\end{proof}

\section{Scope, limitations, and open problems}
\label{sec:open}

The character map in \cref{thm:main} gives a quotient of the mod-$p$ abelianization; it does not determine its kernel or the ordinary abelianization. Moreover, \cref{rem:fixed-section-proper} shows that no one fixed-coordinate section exhausts $Q_n(k)$.

\begin{problem}[Kernel, splitting, and abelianization]
Determine the following.
\begin{enumerate}[label=\textup{(\roman*)}]
\item The kernel of
\[
\chi_{\mathrm{all}}:U_n(k)\longrightarrow
\bigoplus_{s\in S_{n,p}}(k^n,+).
\]
In particular, decide whether
\[
\ker\chi_{\mathrm{all}}
=
[U_n(k),U_n(k)]\,U_n(k)^{[p]}.
\]
Equivalently, is the induced surjection
\[
Q_n(k)\longrightarrow
\bigoplus_{s\in S_{n,p}}(k^n,+)
\]
an isomorphism?

\item Whether $\chi_{\mathrm{all}}$ admits a group-theoretic section. Theorem~\ref{thm:main} supplies sections after fixing one target coordinate, but the explicit lifts belonging to different coordinates need not commute.

\item The ordinary abelianization
\[
U_n(k)^{\mathrm{ab}}
=U_n(k)/[U_n(k),U_n(k)]
\]
and the remaining structure of $Q_n(k)$ if the map in part~\textup{(i)} is not an isomorphism.
\end{enumerate}
\end{problem}

The tame hypothesis enters through \cref{thm:tame-admissibility}: admissibility was verified on special affine and elementary generators and propagated through tame words. \Cref{ex:Cartier-not-iterable} shows that Cartier-admissibility is not automatic for an arbitrary de Rham class.

\begin{problem}[Admissibility beyond the tame group]
Let $F\in\operatorname{GA}_n(k)$ satisfy $\det JF=1$. Under what conditions is the flux class
\[
[F^*\lambda-\lambda]
\]
Cartier-admissible? Is every special polynomial automorphism Cartier-admissible?
\end{problem}

The proof is formulated over finite fields. Perfectness is used in the Cartier coefficient formulas and the polynomial de Rham decomposition; tameness is used in \cref{thm:tame-admissibility}; and finiteness is used in \cref{prop:finite-index}, to make the truncated jet images finite, and in the countability argument proving the exact dimension and the finite-generation classification. Over an infinite field the tangent subgroup need not have finite index in the full tame group. Extending the character construction to other perfect fields requires a separate formulation with the appropriate Frobenius twists and is not asserted here.

The family $\chi_{r,j}$ extracts one Frobenius level at a time. It is natural to ask whether these characters arise from a single intrinsic cohomological object.

\begin{problem}[Intrinsic packaging]
Can the character tower be realized by a natural de Rham--Witt, crystalline, or related object? Can such a refinement detect higher $p$-power information in the abelianization of $U_n(k)$?
\end{problem}

\subsection*{Conclusion}

We have constructed a positive-depth tower of ordinary Cartier characters on
the finite-index tangent subgroup of the tame polynomial automorphism
group. The underlying primitive-form flux, its zeroth Cartier coordinate,
and the inverse-Cartier monomial representative pattern have counterparts
in Ma's recent formal-group work. The new phenomenon is their repeated
tame-subgroup descent and survival at arbitrarily deep positive levels,
together with the cross-level delta family. The joint character map is
surjective onto a countable direct sum, each fixed-coordinate quotient is
split, and the maximal elementary abelian $p$-quotient has countably
infinite dimension. This proves the finite-generation classification.
The kernel of the joint character map, a possible all-coordinate splitting,
the exact mod-$p$ abelianization, the ordinary abelianization, and
admissibility beyond the tame subgroup remain open.

\section*{Declaration of generative AI and AI-assisted technologies}
During preparation of this work, the authors used ChatGPT for exploratory mathematical brainstorming, manuscript organization and language editing, LaTeX assistance, and literature-search support. The authors reviewed, corrected, and independently verified all AI-assisted material and take full responsibility for the content of the article.

\bibliographystyle{amsplain}
\bibliography{Paper_D_final_submission}

\end{document}